\documentclass[11pt,fleqn]{article}

\usepackage{amsmath,amssymb,amsthm,cite}

\usepackage[left=2.8cm,right=2.8cm,top=3.1cm,bottom=3.1cm]{geometry}

\usepackage[colorlinks=true,urlcolor=blue,
citecolor=red,linkcolor=blue,linktocpage,pdfpagelabels,
bookmarksnumbered,bookmarksopen]{hyperref}

\usepackage[hyperpageref]{backref}

\newcommand{\bdry}[1]{\partial #1}
\newcommand{\bgset}[1]{\big\{#1\big\}}
\newcommand{\A}{{\cal A}}
\newcommand{\F}{{\cal F}}
\newcommand{\comp}{\circ}
\newcommand{\dist}[2]{\text{dist}\, (#1,#2)}
\newcommand{\eps}{\varepsilon}
\newcommand{\half}{\frac{1}{2}}
\newcommand{\id}[1][]{id_{\, #1}}
\newcommand{\incl}{\subset}
\newcommand{\N}{\mathbb N}
\newcommand{\norm}[2][]{\left\|#2\right\|_{#1}}
\renewcommand{\O}{{\mathcal O}}
\renewcommand{\o}{\text{o}}
\newcommand{\PS}[1]{$(\text{PS})_{#1}$}
\newcommand{\pnorm}[2][]{\if #1'' \left|#2\right|_p \else \left|#2\right|_{#1} \fi}
\newcommand{\R}{\mathbb R}
\newcommand{\RP}{\R \text{P}}
\newcommand{\restr}[2]{\left.#1\right|_{#2}}
\newcommand{\seq}[1]{\left(#1\right)}
\newcommand{\set}[1]{\left\{#1\right\}}

\newcommand{\wto}{\rightharpoonup}
\newcommand{\Z}{\mathbb Z}

\DeclareMathOperator{\volume}{vol}

\newenvironment{properties}[1]{\begin{enumerate}
\renewcommand{\theenumi}{$(#1_\arabic{enumi})$}
\renewcommand{\labelenumi}{$(#1_\arabic{enumi})$}}{\end{enumerate}}

\newtheorem{lemma}{Lemma}[section]
\newtheorem{proposition}[lemma]{Proposition}
\newtheorem{theorem}[lemma]{Theorem}

\numberwithin{equation}{section}

\title{\bf Critical fractional $p$-Laplacian problems \\ with possibly vanishing potentials\thanks{{\em MSC2010:} Primary 35R11, 35B33, Secondary 58E05
\newline \indent\; {\em Key Words and Phrases:} fractional $p$-Laplacian, critical exponent, external potentials, nontrivial solutions, generalized linking, $\Z_2$-cohomological index}}
\author{\bf Kanishka Perera\\
Department of Mathematical Sciences\\
Florida Institute of Technology\\
Melbourne, FL 32901, USA\\
[\bigskipamount]
\bf Marco Squassina\thanks{Research partially supported
by ``Gruppo Nazionale per l'Analisi Matematica, la Probabilit\`a e le loro Applicazioni (INdAM)'.}\\
Dipartimento di Informatica\\
Universit\`a degli Studi di Verona\\
37134 Verona, Italy\\
[\bigskipamount]
\bf Yang Yang\thanks{Research supported by NSFC-Tian Yuan Special Foundation (No.\ 11226116), Natural Science Foundation of Jiangsu Province of China for Young Scholars (No.\ BK2012109), and the China Scholarship Council (No.\ 201208320435).}\\
School of Science\\
Jiangnan University\\
Wuxi, Jiangsu, 214122, China}
\date{}

\begin{document}

\maketitle

\begin{abstract}
We obtain nontrivial solutions of a critical fractional $p$-Laplacian equation in the whole space and with possibly vanishing potentials. In addition to the usual difficulty of the lack of compactness associated with problems involving critical Sobolev exponents, the problem is further complicated by the absence of a direct sum decomposition suitable for applying classical linking arguments. We overcome this difficulty using a generalized linking construction based on the $\Z_2$-cohomological index.
\end{abstract}

\section{Introduction and main results}

For $p \in (1,\infty)$, $s \in (0,1)$, and $N > sp$, the fractional $p$-Laplacian is the nonlinear nonlocal operator defined on smooth functions by
\begin{equation} \label{1}
(- \Delta)_p^s\, u(x) = 2\, \lim_{\eps \searrow 0}\, \int_{\R^N \setminus B_\eps(x)} \frac{|u(x) - u(y)|^{p-2}\, (u(x) - u(y))}{|x - y|^{N+sp}}\, dy, \quad x \in \R^N.
\end{equation}
This definition is consistent, up to a normalization constant depending on $N$ and $s$, with the usual definition of the linear fractional Laplacian $(- \Delta)^s$ when $p = 2$. Some motivations that have led to the study of these kind of operators can be found in Caffarelli \cite{Ca}.

\noindent
The operator $(- \Delta)_p^s$ leads naturally to the quasilinear problem
\[
\begin{cases}
(- \Delta)_p^s\, u = f(x,u)  & \text{in $\Omega$} \\[2pt]
\, u = 0  & \text{in $\R^N \setminus \Omega$},
\end{cases}
\]
where $\Omega$ is a domain in $\R^N$. There is currently a rapidly growing literature on this problem when $\Omega$ is bounded with Lipschitz boundary. In particular, fractional $p$-eigenvalue problems have been studied in \cite{FrPa,MR3245079,MR3148135,PeSqYa3}, regularity theory in \cite{BP,DiKuPa,IaMoSq}, existence theory in the subcritical case in \cite{IaLiPeSq}, and the critical case in \cite{PeSqYa2}.

\noindent
The corresponding problem in the whole space was recently considered in Lehrer et al.\! \cite{MR3299115} and Torres \cite{To}. In \cite{To}, the equation
\begin{equation} \label{2}
(- \Delta)_p^s\, u + V(x)\, |u|^{p-2}\, u = f(x,u), \quad x \in \R^N
\end{equation}
was studied when the potential $V \in C(\R^N)$ satisfies
\[
\inf_{x \in \R^N}\, V(x) > 0, \qquad \mu\left(\set{x \in \R^N : V(x) \le M}\right) < + \infty \quad \forall M > 0,
\]
where $\mu$ denotes the Lebesgue measure in $\R^N$, and a nontrivial solution was obtained when the nonlinearity $f$ is $p$-superlinear and subcritical.

\noindent
Equation \eqref{2} reduces to the well-known fractional Schr\"{o}dinger equation
\begin{equation} \label{3}
(- \Delta)^s\, u + V(x)\, u = f(x,u)
\end{equation}
when $p = 2$. This equation, introduced by Laskin \cite{MR1755089,MR1948569}, is an important model in fractional quantum mechanics and comes from an expansion of the Feynman path integral from Brownian-like to L\'{e}vy-like quantum mechanical paths. When $s = 1$, the L\'{e}vy dynamics becomes the Brownian dynamics and equation \eqref{3} reduces to the classical Schr\"{o}dinger equation. The fractional Schr\"{o}dinger equation has been widely investigated during the last decade and positive solutions have been obtained under various assumptions on $V$ and $f$ (see, e.g., \cite{MR2953151,MR3002595,MR3112523,MR3059423,MR3156081,
MR3153832,To2,ZhLi} and the references therein).

\noindent
In the present paper we investigate existence of nontrivial solutions of the equation
\begin{equation} \label{4}
(- \Delta)_p^s\, u + V(x)\, |u|^{p-2}\, u = \lambda K(x)\, |u|^{p-2}\, u + \mu L(x)\, |u|^{q-2}\, u + |u|^{p_s^\ast - 2}\, u, \quad x \in \R^N,
\end{equation}
where $p_s^\ast = Np/(N - sp)$ is the fractional critical Sobolev exponent, $q \in (p,p_s^\ast)$, $V,\, K,\, L$ are positive continuous functions, and $\lambda \in \R$,
$\mu > 0$ are parameters. The semi-linear local case $p = 2,\, s = 1$ of this problem has been extensively studied in the literature, where the main feature is to impose conditions on $V$, $K$ and $L$ to gain some compactness (see, e.g., \cite{MR695535,MR2100913,MR2174974,MR2187794,MR2276857,MR2602152,MR3003299}). In the quasilinear case $p \ne 2$, in addition to the usual difficulty of the lack of compactness, this problem is further complicated by the absence of a direct sum decomposition suitable for applying the classical linking theorem of Rabinowitz \cite{MR0488128}. We will overcome this difficulty by using a generalized linking theorem based on the $\Z_2$-cohomological index.
\vskip3pt
\noindent
We shall assume that $V$, $K$ and $L$ are positive continuous functions on $\R^N$ satisfying:
\begin{description}
\item[$(H_1)$] $K$ and $L$ are bounded;
\item[$(H_2)$] we have:
    \[
    \lim_{|x| \to \infty}\, \frac{K(x)}{V(x)} = 0, \qquad \lim_{|x| \to \infty}\, \frac{L(x)}{V(x)^{[Np-(N-sp)q]/sp^2}} = 0.
    \]
\end{description}
Notice that these conditions allow possibly different behaviour of the potentials at infinity. If the potential $V$ is bounded, then $K$ and $L$ need to be
vanishing at infinity, while if $V$ is unbounded, then $K$ and $L$ may converge to a constant. Denoting by
\[
[u]_{s,p} = \left(\int_{\R^{2N}} \frac{|u(x) - u(y)|^p}{|x - y|^{N+sp}}\, dx dy\right)^{1/p}
\]
the Gagliardo seminorm of a measurable function $u : \R^N \to \R$ and setting
\[
\pnorm[p,V]{u} = \left(\int_{\R^N} V(x)\, |u|^p\, dx\right)^{1/p},
\]
we work in the reflexive Banach space
\[
X = \set{u \in L^{p_s^\ast}(\R^N) : [u]_{s,p} < \infty,\, \pnorm[p,V]{u} < \infty}
\]
with the norm given by
\[
\norm{u}^p = [u]_{s,p}^p + \pnorm[p,V]{u}^p.
\]
Let also
\[
L^p_K(\R^N) = \set{u : \int_{\R^N} K(x)\, |u|^p\, dx < \infty}, \qquad L^q_L(\R^N) = \set{u : \int_{\R^N} L(x)\, |u|^q\, dx < \infty}
\]
be the weighted Lebesgue spaces, normed by
\[
\pnorm[p,K]{u} = \left(\int_{\R^N} K(x)\, |u|^p\, dx\right)^{1/p}, \qquad \pnorm[q,L]{u} = \left(\int_{\R^N} L(x)\, |u|^q\, dx\right)^{1/q},
\]
respectively. We have the following compactness result.

\begin{proposition} \label{Proposition 1}
If $(H_1)$ and $(H_2)$ hold, then $X$ is compactly embedded in $L^p_K(\R^N) \cap L^q_L(\R^N)$.
\end{proposition}

\noindent
A weak solution of equation \eqref{4} is a function $u \in X$ satisfying
\begin{multline*}
\int_{\R^{2N}} \frac{|u(x) - u(y)|^{p-2}\, (u(x) - u(y))\, (v(x) - v(y))}{|x - y|^{N+sp}}\, dx dy + \int_{\R^N} V(x)\, |u|^{p-2}\, uv\, dx\\[5pt]
= \int_{\R^N} \left(\lambda K(x)\, |u|^{p-2}\, u + \mu L(x)\, |u|^{q-2}\, u + |u|^{p_s^\ast - 2}\, u\right) v\, dx \quad \forall v \in X.
\end{multline*}
Weak solutions coincide with critical points of the $C^1$-functional
\[
\Phi(u) = \frac{1}{p}\, \norm{u}^p - \int_{\R^N} \left(\frac{\lambda}{p}\, K(x)\, |u|^p + \frac{\mu}{q}\, L(x)\, |u|^q + \frac{1}{p_s^\ast}\, |u|^{p_s^\ast}\right) dx, \quad u \in X.
\]
Recall that $\Phi$ satisfies the Palais-Smale compactness condition at the level $c \in \R$, or \PS{c} for short, if every sequence $\seq{u_j} \subset X$ such that $\Phi(u_j) \to c$ and $\Phi'(u_j) \to 0$, called a \PS{c} sequence, has a convergent subsequence. Let
\begin{equation} \label{5}
S_V = \inf_{u \in X \setminus \set{0}}\, \frac{\norm{u}^p}{\pnorm[p_s^\ast]{u}^p} > 0,
\end{equation}
where $\pnorm[p_s^\ast]{\cdot}$ is the standard norm in $L^{p_s^\ast}(\R^N)$. Our existence results will be based on the following proposition.

\begin{proposition} \label{Proposition 2}
Assume that
$$
c < \dfrac{s}{N}\, S_V^{N/sp},\quad c \ne 0.
$$
Then any {\em \PS{c}} sequence has a subsequence
converging weakly to a nontrivial solution of \eqref{4}.
\end{proposition}

\noindent
Let
\[
S_{s,p} = \inf_{u \in L^{p_s^\ast}(\R^N) \setminus \set{0},\, [u]_{s,p} < \infty}\, \frac{[u]_{s,p}^p}{\pnorm[p_s^\ast]{u}^p}
\]
be the best constant in the fractional Sobolev inequality. Since $V$ is positive, $S_V \ge S_{s,p}$. Equality holds if $V \in L^{N/sp}(\R^N)$. To see this, let $\varphi$ be a minimizer for $S_{s,p}$ and let $y_j \in \R^N$ with $|y_j| \to \infty$. Then $u_j := \varphi(\cdot - y_j) \in X$ by the H\"{o}lder inequality and $\pnorm[p,V]{u_j} \to 0$ as easily seen by approximating $V$ and $\varphi$ by functions with compact supports. Since $[\cdot]_{s,p}$ and $\pnorm[p_s^\ast]{\cdot}$ are translation invariant, then
\[
\frac{\norm{u_j}^p}{\pnorm[p_s^\ast]{u_j}^p} = \frac{[\varphi]_{s,p}^p + \pnorm[p,V]{u_j}^p}{\pnorm[p_s^\ast]{\varphi}^p} \to S_{s,p}.
\]
Moreover, in this case, the infimum in \eqref{5} is not attained. For if $u_0$ is a minimizer for $S_V$, then $S_V = \norm{u_0}^p/\pnorm[p_s^\ast]{u_0}^p > [u_0]_{s,p}^p/\pnorm[p_s^\ast]{u_0}^p \ge S_{s,p}$, a contradiction. It is expected that the minimizers $\varphi$ of $S_{s,p}$ decay at infinity
as $\varphi(x)\thicksim |x|^{-(N-sp)/(p-1)}$ (this occurs for $s=1$, see \cite{talenti}) so that, for
$$
\alpha_0=
\begin{cases}
1 & \text{if $N>sp^2$} \\
\frac{N}{p'(N-sp)}   & \text{if $sp<N\leq sp^2$},
\end{cases}
$$
the conclusion $S_V=S_{s,p}$ is expected to hold by assuming that $V\in L^{\alpha'}(\R^N)$ for some $\alpha>\alpha_0$. \newline
\noindent
Since $X$ is compactly embedded in $L^p_K(\R^N)$ by Proposition \ref{Proposition 1},
\begin{equation} \label{6}
\lambda_1 = \inf_{u \in X \setminus \set{0}}\, \frac{\norm{u}^p}{\pnorm[p,K]{u}^p} > 0
\end{equation}
is the first eigenvalue of the eigenvalue problem
\begin{equation} \label{7}
(- \Delta)_p^s\, u + V(x)\, |u|^{p-2}\, u = \lambda K(x)\, |u|^{p-2}\, u, \quad u \in X.
\end{equation}
First we obtain a positive solution of equation \eqref{4} when $\lambda < \lambda_1$.

\begin{theorem} \label{Theorem 1}
Assume $(H_1)$ and $(H_2)$. If $\lambda < \lambda_1$, then there exists $\mu^\ast(\lambda) > 0$ such that equation \eqref{4} has a positive solution for all $\mu \ge \mu^\ast(\lambda)$.
\end{theorem}

\noindent
Let $u^\pm(x) = \max \set{\pm u(x),0}$ be the positive and negative parts of $u$, respectively, and set
\[
\Phi^+(u) = \frac{1}{p}\, \norm{u}^p - \int_{\R^N} \left(\frac{\lambda}{p}\, K(x)\, (u^+)^p + \frac{\mu}{q}\, L(x)\, (u^+)^q + \frac{1}{p_s^\ast}\, (u^+)^{p_s^\ast}\right) dx, \quad u \in X.
\]
If $u$ is a critical point of $\Phi^+$, then recalling the elementary inequality
$$
|a^--b^-|^p\leq |a-b|^{p-2}(a-b)(b^--a^-),\quad\, \forall a,b\in\R,
$$
we obtain
\begin{align*}
0 = {\Phi^+}'(u)\, (- u^-) &=
\int_{\R^{2N}} \frac{|u(x) - u(y)|^{p-2}(u(x)-u(y))(u^-(y)-u^-(x))}{|x - y|^{N+sp}}\, dx dy\\
&+\int_{\R^N}V(x)|u^-|^pdx\geq \|u^-\|^p
\end{align*}
and hence $u^- = 0$, so $u = u^+ \ge 0$ is a critical point of $\Phi$ and therefore a nonnegative solution of equation \eqref{4}. Moreover, if it was $u(x_0) = 0$ for some $x_0 \in \R^N$, then \eqref{4} and \eqref{1} give
\[
\lim_{\eps \searrow 0}\, \int_{\R^N \setminus B_\eps(x_0)} \frac{u(y)^{p-1}}{|x_0 - y|^{N+sp}}\, dy = 0,
\]
so $u = 0$. Thus, nontrivial critical points of $\Phi^+$ are positive solutions of \eqref{4}. The proof of Theorem \ref{Theorem 1} will be based on constructing a minimax level of mountain pass type for $\Phi^+$ below the threshold level given in Proposition \ref{Proposition 2}.

\noindent
Next we obtain a (possibly nodal) nontrivial solution of equation \eqref{4} when $\lambda \ge \lambda_1$.

\begin{theorem} \label{Theorem 2}
Assume $(H_1)$ and $(H_2)$. If $\lambda \ge \lambda_1$, then there exists $\mu_\ast(\lambda) > 0$ such that equation \eqref{4} has a nontrivial solution for all $\mu \ge \mu_\ast(\lambda)$.
\end{theorem}

\noindent
This extension of Theorem \ref{Theorem 1} is nontrivial. Indeed, the functional $\Phi$ does not have the mountain pass geometry when $\lambda \ge \lambda_1$ since the origin is no longer a local minimizer, and a linking type argument is needed. However, the classical linking theorem cannot be used since the nonlinear eigenvalue problem \eqref{7} does not have linear eigenspaces.

\section{Tools from critical point theory}
We will use a general construction based on sublevel sets as in Perera and Szulkin \cite{MR2153141} (see also Perera et al.\! \cite[Proposition 3.23]{MR2640827}). Moreover, the standard sequence of eigenvalues of \eqref{7} based on the genus does not give enough information about the structure of the sublevel sets to carry out this linking construction. Therefore we will use a different sequence of eigenvalues as in Perera \cite{MR1998432} that is based on a cohomological index.

\noindent
The $\Z_2$-cohomological index of Fadell and Rabinowitz \cite{MR57:17677} is defined as follows. Let $W$ be a Banach space and let $\A$ denote the class of symmetric subsets of $W \setminus \set{0}$. For $A \in \A$, let $\overline{A} = A/\Z_2$ be the quotient space of $A$ with each $u$ and $-u$ identified, let $f : \overline{A} \to \RP^\infty$ be the classifying map of $\overline{A}$, and let $f^\ast : H^\ast(\RP^\infty) \to H^\ast(\overline{A})$ be the induced homomorphism of the Alexander-Spanier cohomology rings. The cohomological index of $A$ is defined by
\[
i(A) = \begin{cases}
\sup \set{m \ge 1 : f^\ast(\omega^{m-1}) \ne 0}, & A \ne \emptyset\\[5pt]
0, & A = \emptyset,
\end{cases}
\]
where $\omega \in H^1(\RP^\infty)$ is the generator of the polynomial ring $H^\ast(\RP^\infty) = \Z_2[\omega]$. For example, the classifying map of the unit sphere $S^{m-1}$ in $\R^m,\, m \ge 1$ is the inclusion $\RP^{m-1} \incl \RP^\infty$, which induces isomorphisms on $H^q$ for $q \le m - 1$, so $i(S^{m-1}) = m$. The following proposition summarizes the basic properties of this index.

\begin{proposition}[Fadell-Rabinowitz \cite{MR57:17677}] \label{Proposition 3}
The index $i : \A \to \N \cup \set{0,\infty}$ has the following properties:
\begin{properties}{i}
\item Definiteness: $i(A) = 0$ if and only if $A = \emptyset$;
\item \label{i2} Monotonicity: If there is an odd continuous map from $A$ to $B$ (in particular, if $A \subset B$), then $i(A) \le i(B)$. Thus, equality holds when the map is an odd homeomorphism;
\item Dimension: $i(A) \le \dim W$;
\item Continuity: If $A$ is closed, then there is a closed neighborhood $N \in \A$ of $A$ such that $i(N) = i(A)$. When $A$ is compact, $N$ may be chosen to be a $\delta$-neighborhood $N_\delta(A) = \set{u \in W : \dist{u}{A} \le \delta}$;
\item Subadditivity: If $A$ and $B$ are closed, then $i(A \cup B) \le i(A) + i(B)$;
\item Stability: If $SA$ is the suspension of $A \ne \emptyset$, obtained as the quotient space of $A \times [-1,1]$ with $A \times \set{1}$ and $A \times \set{-1}$ collapsed to different points, then $i(SA) = i(A) + 1$;
\item Piercing property: If $A$, $A_0$ and $A_1$ are closed, and $\varphi : A \times [0,1] \to A_0 \cup A_1$ is a continuous map such that $\varphi(-u,t) = - \varphi(u,t)$ for all $(u,t) \in A \times [0,1]$, $\varphi(A \times [0,1])$ is closed, $\varphi(A \times \set{0}) \subset A_0$ and $\varphi(A \times \set{1}) \subset A_1$, then $i(\varphi(A \times [0,1]) \cap A_0 \cap A_1) \ge i(A)$;
\item Neighborhood of zero: If $U$ is a bounded closed symmetric neighborhood of $0$, then $i(\bdry{U}) = \dim W$.
\end{properties}
\end{proposition}

\noindent
Eigenvalues of problem \eqref{7} coincide with critical values of the functional
\[
\Psi(u) = \left(\int_{\R^N} K(x)\, |u|^p\, dx\right)^{-1}
\]
on the unit sphere $S = \set{u \in X : \norm{u} = 1}$, and we can define an increasing and unbounded sequence of eigenvalues via a suitable minimax scheme. The standard scheme based on the genus does not give the index information necessary to prove Theorem \ref{Theorem 2}, so we will use a different scheme based on a cohomological index as in Perera \cite{MR1998432}. First we note that the general theory developed in Perera et al.\! \cite{MR2640827} applies to this problem. Indeed, the odd $(p - 1)$-homogeneous operator $A_p \in C(X,X^\ast)$, where $X^\ast$ is the dual of $X$, defined by
\begin{multline} \label{19}
A_p(u)\, v = \int_{\R^{2N}} \frac{|u(x) - u(y)|^{p-2}\, (u(x) - u(y))\, (v(x) - v(y))}{|x - y|^{N+sp}}\, dx dy\\[5pt]
+ \int_{\R^N} V(x)\, |u|^{p-2}\, uv\, dx, \quad u, v \in X
\end{multline}
that is associated with the left-hand side of equation \eqref{7} is the Fr\'{e}chet derivative of the $C^1$-functional $X \to \R,\, u \mapsto \norm{u}^p\!/p$ and satisfies for all $u, v \in X$,
\begin{equation} \label{8}
A_p(u)\, u = \norm{u}^p
\end{equation}
and
\begin{multline} \label{27}
|A_p(u)\, v| \le \int_{\R^{2N}} \frac{|u(x) - u(y)|^{p-1}\, |v(x) - v(y)|}{|x - y|^{N+sp}}\, dx dy + \int_{\R^N} V(x)\, |u|^{p-1}\, |v|\, dx\\[5pt]
\le [u]_{s,p}^{p-1}\, [v]_{s,p} + \pnorm[p,V]{u}^{p-1} \pnorm[p,V]{v} \le \norm{u}^{p-1} \norm{v}
\end{multline}
by the H\"{o}lder inequalities for integrals and sums. Moreover, since $X$ is uniformly convex, it follows from \eqref{8} and \eqref{27} that $A_p$ is of type (S), i.e., every sequence $\seq{u_j} \subset X$ such that
\[
u_j \wto u, \quad A_p(u_j)\, (u_j - u) \to 0
\]
has a subsequence that converges strongly to $u$ (cf.\ \cite[Proposition 1.3]{MR2640827}). Hence the operator $A_p$ satisfies the structural assumptions of \cite[Chapter 1]{MR2640827}. On the other hand, the odd $(p - 1)$-homogeneous operator $B_p \in C(X,X^\ast)$ defined by
\[
B_p(u)\, v = \int_{\R^N} K(x)\, |u|^{p-2}\, uv\, dx, \quad u, v \in X
\]
that appears in the right-hand side of \eqref{7} is the Fr\'{e}chet derivative of the $C^1$-functional $X \to \R,\, u \mapsto \pnorm[p,K]{u}^p\!/p$ and satisfies $B_p(u)\, u = \pnorm[p,K]{u}^p$ for all $u \in X$. Moreover, since $X$ is compactly embedded in $L^p_K(\R^N)$ by Proposition \ref{Proposition 1}, $B_p$ is compact. Hence $B_p$ satisfies the assumptions of \cite[Chapter 4]{MR2640827}. Let $\F$ denote the class of symmetric subsets of $S$ and set
\[
\lambda_k := \inf_{M \in \F,\; i(M) \ge k}\, \sup_{u \in M}\, \Psi(u), \quad k \in \N.
\]
By \cite[Theorem 4.6]{MR2640827}, $\lambda_k \nearrow + \infty$ is a sequence of eigenvalues of problem \eqref{7} and
\begin{equation} \label{9}
\lambda_k < \lambda_{k+1} \implies i(\Psi^{\lambda_k}) = i(S \setminus \Psi_{\lambda_{k+1}}) = k,
\end{equation}
where $\Psi^{\lambda_k} = \set{u \in S : \Psi(u) \le \lambda_k}$ and $\Psi_{\lambda_{k+1}} = \set{u \in S : \Psi(u) \ge \lambda_{k+1}}$.

\noindent
The proof of Theorem \ref{Theorem 2} will make essential use of \eqref{9} and will be based on the following abstract critical point theorem.

\begin{theorem} \label{Theorem 3}
Let $X$ be a Banach space and let $S = \set{u \in X : \norm{u} = 1}$ be the unit sphere in $X$. Let $\Phi$ be a $C^1$-functional on $X$ and let $A_0,\, B_0$ be disjoint nonempty closed symmetric subsets of $S$ such that
\begin{equation} \label{10}
i(A_0) = i(S \setminus B_0) < \infty.
\end{equation}
Assume that there exist $R > r > 0$ and $v \in S \setminus A_0$ such that
\[
\sup \Phi(A) \le \inf \Phi(B), \qquad \sup \Phi(D) < \infty,
\]
where
\begin{gather*}
A = \set{tu : u \in A_0,\, 0 \le t \le R} \cup \set{R\, \pi((1 - t)\, u + tv) : u \in A_0,\, 0 \le t \le 1},\\[5pt]
B = \set{ru : u \in B_0},\\[5pt]
D = \set{tu : u \in A,\, \norm{u} = R,\, 0 \le t \le 1},
\end{gather*}
and $\pi : X \setminus \set{0} \to S,\, u \mapsto u/\!\norm{u}$ is the radial projection onto $S$. Let
$$
\Gamma = \{\gamma \in C(D,X) : \gamma(D) \text{ is closed and} \restr{\gamma}{A} = \id[\!A]\}
$$
and set
\[
c := \inf_{\gamma \in \Gamma}\, \sup_{u \in \gamma(D)}\, \Phi(u).
\]
Then $\inf \Phi(B) \le c \le \sup \Phi(D)$ and $\Phi$ has a {\em \PS{c}} sequence.
\end{theorem}

\noindent
Theorem \ref{Theorem 3}, which does not require a direct sum decomposition, generalizes the linking theorem of Rabinowitz \cite{MR0488128} and is proved in Candito et al.\! \cite{CaMaPe} (see also Yang and Perera \cite{YaPe2}). The linking construction in its proof was also used in Perera and Szulkin \cite{MR2153141} to obtain nontrivial solutions of $p$-Laplacian problems with nonlinearities that interact with the spectrum. A similar construction based on the notion of cohomological linking was given in Degiovanni and Lancelotti \cite{MR2371112}. See also Perera et al.\! \cite[Proposition 3.23]{MR2640827}.

\section{Preliminaries}

In this preliminary section we prove
Propositions \ref{Proposition 1} and \ref{Proposition 2}.

\begin{proof}[Proof of Proposition \ref{Proposition 1}]
Let $\seq{u_j}$ be a bounded sequence in $X$. Then, a renamed subsequence converges weakly and a.e.\ to some $u \in X$. By virtue of assumption $(H_2)$, given $\eps > 0$, there exists $r_\eps > 0$ such that
$$
K(x) \le \eps\, V(x),\quad\,\,
L(x) \le \eps\, V(x)^{[Np-(N-sp)q]/sp^2},
\qquad\text{for all $x \in B_{r_\eps}(0)^c$}.
$$
The first inequality gives
\begin{multline} \label{11}
\int_{B_{r_\eps}(0)^c} K(x)\, |u_j|^p\, dx \le \eps \int_{\R^N} V(x)\, |u_j|^p\, dx = \eps\, \O(1),\\[5pt]
\int_{B_{r_\eps}(0)^c} K(x)\, |u|^p\, dx \le \eps \int_{\R^N} V(x)\, |u|^p\, dx.
\end{multline}
Combining the second inequality with the Young's inequality gives
$$
L(x)\, |t|^q \le \eps \big(V(x)\, |t|^p + |t|^{p_s^\ast}\big),\qquad \text{for all $x \in B_{r_\eps}(0)^c$ and $t \in \R$},
$$
so that
\begin{multline} \label{12}
\int_{B_{r_\eps}(0)^c} L(x)\, |u_j|^q\, dx \le \eps \int_{\R^N} \left(V(x)\, |u_j|^p + |u_j|^{p_s^\ast}\right) dx = \eps\, \O(1),\\[5pt]
\int_{B_{r_\eps}(0)^c} L(x)\, |u|^q\, dx \le \eps \int_{\R^N} \left(V(x)\, |u|^p + |u|^{p_s^\ast}\right) dx.
\end{multline}
Since $X$ is compactly embedded in $L^p(B_{r_\eps}(0)) \cap L^q(B_{r_\eps}(0))$ and $K$ and $L$ are bounded, $u_j \to u$ in $L^p_K(B_{r_\eps}(0)) \cap L^q_L(B_{r_\eps}(0))$ for a further subsequence. Then it follows from \eqref{11} and \eqref{12} that $u_j \to u$ in $L^p_K(\R^N) \cap L^q_L(\R^N)$, since they are uniformly convex spaces.
\end{proof}

\begin{proof}[Proof of Proposition \ref{Proposition 2}]
Let $\seq{u_j}$ be a \PS{c} sequence. First we show that $\seq{u_j}$ is bounded in $X$. We have
\begin{equation} \label{13}
\Phi(u_j) = \frac{1}{p}\, \norm{u_j}^p - \int_{\R^N} \left(\frac{\lambda}{p}\, K(x)\, |u_j|^p + \frac{\mu}{q}\, L(x)\, |u_j|^q + \frac{1}{p_s^\ast}\, |u_j|^{p_s^\ast}\right) dx = c + \o(1)
\end{equation}
and
\begin{multline} \label{28}
\Phi'(u_j)\, v = \int_{\R^{2N}} \frac{|u_j(x) - u_j(y)|^{p-2}\, (u_j(x) - u_j(y))\, (v(x) - v(y))}{|x - y|^{N+sp}}\, dx dy\\[5pt]
+ \int_{\R^N} V(x)\, |u_j|^{p-2}\, u_j\, v\, dx - \int_{\R^N} \left(\lambda K(x)\, |u_j|^{p-2}\, u_j + \mu L(x)\, |u_j|^{q-2}\, u_j + |u_j|^{p_s^\ast - 2}\, u_j\right) v\, dx\\[5pt]
= \o(\norm{v}) \quad \forall v \in X,
\end{multline}
in particular,
\begin{equation} \label{14}
\Phi'(u_j)\, u_j = \norm{u_j}^p - \int_{\R^N} \left(\lambda K(x)\, |u_j|^p + \mu L(x)\, |u_j|^q + |u_j|^{p_s^\ast}\right) dx = \o(\norm{u_j}).
\end{equation}
By \eqref{13} and \eqref{14},
\[
\mu \left(\frac{1}{p} - \frac{1}{q}\right) \int_{\R^N} L(x)\, |u_j|^q\, dx + \left(\frac{1}{p} - \frac{1}{p_s^\ast}\right) \int_{\R^N} |u_j|^{p_s^\ast}\, dx = \o(\norm{u_j}) + \O(1),
\]
which gives
\begin{equation} \label{15}
\int_{\R^N} L(x)\, |u_j|^q\, dx \le \o(\norm{u_j}) + \O(1), \qquad \int_{\R^N} |u_j|^{p_s^\ast}\, dx \le \o(\norm{u_j}) + \O(1)
\end{equation}
since $\mu > 0$ and $p < q < p_s^\ast$. By $(H_2)$, there exists $r > 0$ such that $\lambda K(x) \le V(x)/2$ for all $x \in B_r(0)^c$ and hence
\begin{equation} \label{16}
\int_{B_r(0)^c} \lambda K(x)\, |u_j|^p\, dx \le \half \int_{\R^N} V(x)\, |u_j|^p\, dx \le \half \norm{u_j}^p,
\end{equation}
and by the H\"{o}lder inequality,
\begin{equation} \label{17}
\int_{B_r(0)} \lambda K(x)\, |u_j|^p\, dx \le |\lambda| \pnorm[\infty]{K} \volume\, (B_r(0))^{sp/N} \left(\int_{\R^N} |u_j|^{p_s^\ast}\, dx\right)^{p/p_s^\ast},
\end{equation}
where $\pnorm[\infty]{\cdot}$ is the norm in $L^\infty(\R^N)$. Since $p > 1$, it follows from \eqref{14}--\,\eqref{17} that $\seq{u_j}$ is bounded.
So a renamed subsequence of $\seq{u_j}$ converges to some $u$ weakly in $X$, strongly in $L^p_K(\R^N) \cap L^q_L(\R^N)$ by Proposition \ref{Proposition 1}, and a.e.\! in $\R^N$. The sequence 
$
|u_j(x) - u_j(y)|^{p-2}\, (u_j(x) - u_j(y))/|x - y|^{(N+sp)/p'}
$ 
is bounded in $L^{p'}(\R^{2N})$ and it converges to $|u(x) - u(y)|^{p-2}\, (u(x) - u(y))/|x - y|^{(N+sp)/p'}$ almost everywhere in $\R^{2N}$. 
Moreover, $(v(x) - v(y))/|x - y|^{(N+sp)/p} \in L^p(\R^{2N})$, so the first integral in formula \eqref{28} converges to
\[
\int_{\R^{2N}} \frac{|u(x) - u(y)|^{p-2}\, (u(x) - u(y))\, (v(x) - v(y))}{|x - y|^{N+sp}}\, dx dy
\]
for a subsequence. Similarly,
\begin{gather*}
\int_{\R^N} V(x)\, |u_j|^{p-2}\, u_j\, v\, dx \to \int_{\R^N} V(x)\, |u|^{p-2}\, uv\, dx,\\[5pt]
\int_{\R^N} K(x)\, |u_j|^{p-2}\, u_j\, v\, dx \to \int_{\R^N} K(x)\, |u|^{p-2}\, uv\, dx,\\[5pt]
\int_{\R^N} L(x)\, |u_j|^{q-2}\, u_j\, v\, dx \to \int_{\R^N} L(x)\, |u|^{q-2}\, uv\, dx,\\[5pt]
\int_{\R^N} |u_j|^{p_s^\ast - 2}\, u_j\, v\, dx \to \int_{\R^N} |u|^{p_s^\ast - 2}\, uv\, dx
\end{gather*}
for a further subsequence. So passing to the limit in \eqref{28} shows that $u$ is a weak solution of equation \eqref{4}.
Suppose now that $u = 0$. Since $\seq{u_j}$ is bounded in $X$ and it converges to $0$ in $L^p_K(\R^N) \cap L^q_L(\R^N)$, \eqref{14} and \eqref{5} give
\[
\o(1) = \norm{u_j}^p - \int_{\R^N} |u_j|^{p_s^\ast}\, dx \ge \norm{u_j}^p \left(1 - \frac{\norm{u_j}^{p_s^\ast - p}}{S_V^{p_s^\ast/p}}\right).
\]
If $\norm{u_j} \to 0$, then $\Phi(u_j) \to 0$, contradicting $c \ne 0$, so this implies
\[
\norm{u_j}^p \ge S_V^{N/sp} + \o(1)
\]
for a further subsequence. Then \eqref{13} and \eqref{14} give
\[
c = \left(\frac{1}{p} - \frac{1}{p_s^\ast}\right) \norm{u_j}^p + \o(1) \ge \frac{s}{N}\, S_V^{N/sp} + \o(1),
\]
contradicting $c < \dfrac{s}{N}\, S_V^{N/sp}$.
This concludes the proof.
\end{proof}

\section{Proofs}

In this section we prove Theorems \ref{Theorem 1} and \ref{Theorem 2}.

\begin{proof}[Proof of Theorem \ref{Theorem 1}]
Fix $u_0 > 0$ in $X$ such that $\pnorm[p_s^\ast]{u_0} = 1$. Since $p < q < p_s^\ast$,
\[
\Phi^+(t u_0) = \frac{t^p}{p}\, \norm{u_0}^p - \frac{\lambda\, t^p}{p}\, \pnorm[p,K]{u_0}^p - \frac{\mu\, t^q}{q}\, \pnorm[q,L]{u_0}^q - \frac{t^{p_s^\ast}}{p_s^\ast} \to - \infty
\]
as $t \to + \infty$. Take $t_0 > 0$ so large that $\Phi^+(t_0 u_0) \le 0$, let
\[
\Gamma = \set{\gamma \in C([0,1],X) : \gamma(0) = 0,\, \gamma(1) = t_0 u_0}
\]
be the class of paths joining $0$ and $t_0 u_0$, and set
\[
c := \inf_{\gamma \in \Gamma}\, \max_{u \in \gamma([0,1])}\, \Phi^+(u).
\]
By Proposition \ref{Proposition 1}, we learn that
\begin{equation} \label{18}
T = \inf_{u \in X \setminus \set{0}}\, \frac{\norm{u}^p}{\pnorm[q,L]{u}^p} > 0.
\end{equation}
By formulas \eqref{6}, \eqref{18}, and \eqref{5}, we obtain
\[
\Phi^+(u) \ge \frac{1}{p} \left(1 - \frac{\lambda^+}{\lambda_1}\right) \norm{u}^p - \frac{\mu}{q}\, T^{-q/p} \norm{u}^q - \frac{1}{p_s^\ast}\, S_V^{- p_s^\ast/p} \norm{u}^{p_s^\ast} \quad \forall u \in X,
\]
where $\lambda^+ = \max \set{\lambda,0}$. Since $\lambda^+ < \lambda_1$ and $p_s^\ast > q > p$, it follows from this that $0$ is a strict local minimizer of $\Phi^+$, so $c > 0$. Thus, $\Phi^+$ has a \PS{c} sequence $\seq{u_j}$ by the Mountain Pass Theorem. Since $\gamma(s) = s t_0 u_0$ is a path in $\Gamma$,
\begin{multline*}
c \le \max_{s \in [0,1]}\, \Phi^+(s t_0 u_0) \le \max_{t \ge 0}\, \Phi^+(t u_0) \le \max_{t \ge 0}\, \left[\frac{t^p}{p} \left(\norm{u_0}^p - \lambda \pnorm[p,K]{u_0}^p\right) - \frac{\mu\, t^q}{q}\, \pnorm[q,L]{u_0}^q\right]\\[5pt]
= \left(\frac{1}{p} - \frac{1}{q}\right) \frac{\left(\norm{u_0}^p - \lambda \pnorm[p,K]{u_0}^p\right)^{q/(q-p)}}{\left(\mu \pnorm[q,L]{u_0}^q\right)^{p/(q-p)}} < \frac{s}{N}\, S_V^{N/sp}
\end{multline*}
if $\mu > 0$ is sufficiently large. An argument similar to that in the proof of Proposition \ref{Proposition 2} now shows that a subsequence of $\seq{u_j}$ converges weakly to a positive solution of equation \eqref{4}.
\end{proof}

\noindent
Turning to the proof of Theorem \ref{Theorem 2}, since $\lambda \ge \lambda_1$,
$\lambda_k \le \lambda < \lambda_{k+1}$ for some $k \ge 1$. By \eqref{9},
$i(\Psi^{\lambda_k}) = i(S \setminus \Psi_{\lambda_{k+1}}) = k.$
First, we construct a compact symmetric subset $A_0$ of $\Psi^{\lambda_k}$ with the same index. As we have already noted, the operator $A_p$ defined in \eqref{19} satisfies the structural assumptions of \cite[Chapter 1]{MR2640827}.

\begin{lemma}
The operator $A_p$ is strictly monotone, i.e.,
\[
(A_p(u) - A_p(v))\, (u - v) > 0
\]
for all $u \ne v$ in $X$.
\end{lemma}

\begin{proof}
It is easily seen from \eqref{27} that $A_p(u)\, v \le \norm{u}^{p-1} \norm{v}$ for all $u, v \in X$ and the equality holds if and only if $\alpha u = \beta v$ a.e.\! in $\R^N$ for some $\alpha, \beta \ge 0$, not both zero, so the conclusion follows from \cite[Lemma 6.3]{MR2640827}.
\end{proof}

\begin{lemma} \label{Lemma 6}
For each $w \in L^p_K(\R^N)$, the equation
\begin{equation} \label{29}
(- \Delta)_p^s\, u + V(x)\, |u|^{p-2}\, u = K(x)\, |w|^{p-2}\, w, \quad x \in \R^N
\end{equation}
admits a unique weak solution $u \in X$. Furthermore, the mapping $J : L^p_K(\R^N) \to X,\, w \mapsto u$ is continuous.
\end{lemma}

\begin{proof}
The existence follows from a standard minimization argument since $X$ is continuously embedded in $L^p_K(\R^N)$ by Proposition \ref{Proposition 1}, and the uniqueness is immediate from the strict monotonicity of the operator $A_p$. Let $w_j \to w$ in $L^p_K(\R^N)$ and let $u_j = J(w_j)$, so
\begin{equation} \label{30}
A_p(u_j)\, v = \int_{\R^N} K(x)\, |w_j|^{p-2}\, w_j\, v\, dx \quad \forall v \in X.
\end{equation}
Testing with $v = u_j$ gives
\[
\norm{u_j}^p = \int_{\R^N} K(x)\, |w_j|^{p-2}\, w_j\, u_j\, dx \le \pnorm[p,K]{w_j}^{p-1} \pnorm[p,K]{u_j}
\]
by the H\"{o}lder inequality, which together with the continuity of the embedding $X \hookrightarrow L^p_K(\R^N)$ shows that $\seq{u_j}$ is bounded in $X$. So a renamed subsequence of $\seq{u_j}$ converges to some $u$ weakly in $X$, strongly in $L^p_K(\R^N)$, and a.e.\! in $\R^N$. As in the proof of Proposition \ref{Proposition 2},
\[
A_p(u_j)\, v \to A_p(u)\, v, \qquad \int_{\R^N} K(x)\, |w_j|^{p-2}\, w_j\, v\, dx \to \int_{\R^N} K(x)\, |w|^{p-2}\, wv\, dx,
\]
along a subsequence. So passing to the limit in \eqref{30} shows that $u$ is a weak solution of equation \eqref{29} and hence $u = J(w)$. Testing \eqref{30} with $u_j - u$ gives
\[
A_p(u_j)\, (u_j - u) = \int_{\R^N} K(x)\, |w_j|^{p-2}\, w_j\, (u_j - u)\, dx \le \pnorm[p,K]{w_j}^{p-1} \pnorm[p,K]{u_j - u} \to 0,
\]
so $u_j \to u$ for a further subsequence since $A_p$ is of type (S).
\end{proof}

\begin{proposition} \label{Proposition 4}
If $\lambda_k < \lambda_{k+1}$, then $\Psi^{\lambda_k}$ has a compact symmetric subset $A_0$ with $i(A_0) = k$.
\end{proposition}

\begin{proof}
Let
\[
\pi_{p,K}(u) = \frac{u}{\pnorm[p,K]{u}}, \quad u \in X \setminus \set{0},
\]
be the radial projection onto $S_{p,K} = \bgset{u \in X : \pnorm[p,K]{u} = 1}$, and let
\[
A = \pi_{p,K}(\Psi^{\lambda_k}) = \bgset{w \in S_{p,K} : \norm{w}^p \le \lambda_k}.
\]
Then $i(A) = i(\Psi^{\lambda_k}) = k$ by \ref{i2} of Proposition \ref{Proposition 3} and \eqref{9}. For $w \in A$, let $u = J(w)$, where $J$ is the map defined in Lemma \ref{Lemma 6}, so
\[
A_p(u)\, v = \int_{\R^N} K(x)\, |w|^{p-2}\, wv\, dx, \quad \forall v \in X.
\]
Testing with $v = u, w$ and using the H\"{o}lder inequality gives
\[
\norm{u}^p \le \pnorm[p,K]{w}^{p-1} \pnorm[p,K]{u} = \pnorm[p,K]{u}, \qquad 1 = A_p(u)\, w \le \norm{u}^{p-1} \norm{w},
\]
so
\[
\norm{\pi_{p,K}(u)} = \frac{\norm{u}}{\pnorm[p,K]{u}} \le \norm{w}
\]
and hence $\pi_{p,K}(u) \in A$. Let $\widetilde{J} = \pi_{p,K} \comp J$ and let $\widetilde{A} = \widetilde{J}(A) \subset A$. Since the embedding $X \hookrightarrow L^p_K(\R^N)$ is compact by Proposition \ref{Proposition 1} and $\widetilde{J}$ is an odd continuous map from $L^p_K(\R^N)$ to $X$, then $\widetilde{A}$ is a compact set and $i(\widetilde{A}) = i(A) = k$. Let
\begin{equation} \label{26}
\pi(u) = \frac{u}{\norm{u}}, \quad u \in X \setminus \set{0}
\end{equation}
be the radial projection onto $S$ and let $A_0 = \pi(\widetilde{A})$. Therefore $A_0 \subset \Psi^{\lambda_k}$ is compact and $i(A_0) = i(\widetilde{A}) = k$.
\end{proof}

\noindent
We are now ready to prove Theorem \ref{Theorem 2}.

\begin{proof}[Proof of Theorem \ref{Theorem 2}]
We apply Theorem \ref{Theorem 3}. Since $\lambda \ge \lambda_1$, $\lambda_k \le \lambda < \lambda_{k+1}$ for some $k \ge 1$. Then $\Psi^{\lambda_k}$ has a compact symmetric subset $A_0$ with
\[
i(A_0) = k
\]
by Proposition \ref{Proposition 4}. We take $B_0 = \Psi_{\lambda_{k+1}}$, so that
\[
i(S \setminus B_0) = k
\]
by \eqref{9}. So \eqref{10} holds. For $u \in S$ and $t \ge 0$,
\begin{equation} \label{23}
\Phi(tu) = \frac{t^p}{p} \left(1 - \frac{\lambda}{\Psi(u)}\right) - \frac{\mu\, t^q}{q}\, \pnorm[q,L]{u}^q - \frac{t^{p_s^\ast}}{p_s^\ast}\, \pnorm[p_s^\ast]{u}^{p_s^\ast}.
\end{equation}
Pick any $v \in S \setminus A_0$. Since $A_0$ is compact, so is the set
\[
X_0 = \set{\pi((1 - t)\, u + tv) : u \in A_0,\, 0 \le t \le 1},
\]
where $\pi$ is as in \eqref{26}, and hence
\[
\alpha = \sup_{u \in X_0} \left(1 - \frac{\lambda}{\Psi(u)}\right) < \infty, \qquad \beta = \inf_{u \in X_0}\, \pnorm[q,L]{u}^q > 0.
\]
For $u \in A_0$, \eqref{23} gives
\begin{equation} \label{24}
\Phi(tu) \le - \frac{t^p}{p} \left(\frac{\lambda}{\lambda_k} - 1\right) \le 0
\end{equation}
since $\lambda \ge \lambda_k$. For $u \in X_0$, \eqref{23} gives
\begin{equation} \label{25}
\Phi(tu) \le \frac{\alpha\, t^p}{p} - \frac{\mu \beta\, t^q}{q} \le \left(\frac{1}{p} - \frac{1}{q}\right) \frac{(\alpha^+)^{q/(q-p)}}{(\mu \beta)^{p/(q-p)}},
\end{equation}
where $\alpha^+ = \max \set{\alpha,0}$. Fix $\mu > 0$ so large that the last expression is less than $(s/N) S_V^{N/sp}$, take positive $R \ge (q\, \alpha^+/p\, \mu \beta)^{1/(q-p)}$, and let $A$ and $D$ be as in Theorem \ref{Theorem 3}. Then it follows from \eqref{24} and \eqref{25} that
\[
\sup \Phi(A) \le 0, \qquad \sup \Phi(D) < \frac{s}{N}\, S_V^{N/sp}.
\]
Finally for $u \in B_0$, \eqref{23} gives
\[
\Phi(tu) \ge \frac{t^p}{p} \left(1 - \frac{\lambda}{\lambda_{k+1}}\right) - \frac{\mu\, t^q}{q}\, T^{-q/p} - \frac{t^{p_s^\ast}}{p_s^\ast}\, S_V^{- p_s^\ast/p},
\]
where $T$ is as in \eqref{18}. Since $\lambda < \lambda_{k+1}$ and $p_s^\ast > q > p$, it follows from this that if $0 < r < R$ is sufficiently small and $B$ is as in Theorem \ref{Theorem 3}, then
\[
\inf \Phi(B) > 0.
\]
Thus, $0 < c < (s/N) S_V^{N/sp}$ and $\Phi$ has a \PS{c} sequence by Theorem \ref{Theorem 3}, a subsequence of which converges weakly to a nontrivial solution of equation \eqref{4} by Proposition \ref{Proposition 2}.
\end{proof}

\medskip

\def\cdprime{$''$}

\end{document}